\documentclass[10pt,a4paper]{article}

\usepackage{geometry}

\usepackage{amsmath,amsthm,amsfonts,amssymb,amscd,cite,graphicx}
\usepackage{mathptmx}      

\newtheorem{theorem}{Theorem}
\newtheorem{lemma}{Lemma}

\title{Saturation of multidimensional 0-1 matrices}
\author{Shen-Fu Tsai}

\begin{document}
\maketitle

\begin{abstract}
A 0-1 matrix $M$ is \textit{saturating} for a 0-1 matrix $P$ if $M$ does not contain a submatrix that can be turned into $P$ by flipping any number of its $1$-entries to $0$-entries, and changing any $0$-entry to $1$-entry of $M$ introduces a copy of $P$. Matrix $M$ is \textit{semisaturating} for $P$ if changing any $0$-entry to $1$-entry of $M$ introduces a \textit{new} copy of $P$, regardless of whether $M$ originally contains $P$ or not. The functions $ex(n;P)$ and $sat(n;P)$ are the maximum and minimum possible number of $1$-entries a $n\times n$ 0-1 matrix saturating for $P$ can have, respectively. Function $ssat(n;P)$ is the minimum possible number of $1$-entries a $n\times n$ 0-1 matrix semisaturating for $P$ can have.\\

Function $ex(n;P)$ has been studied for decades, while investigation on $sat(n;P)$ and $ssat(n;P)$ was initiated recently. In this paper, we make nontrivial generalization of results regarding these functions to multidimensional 0-1 matrices. In particular, we find the exact values of $ex(n;P,d)$ and $sat(n;P,d)$ when $P$ is a $d$-dimensional identity matrix. Then we give the necessary and sufficient condition for  a multidimensional 0-1 matrix to have bounded  semisaturation function.\\



\end{abstract}

\section{Introduction}
Extremal combinatorics on pattern avoidance is a central topic in graph theory and combinatorics. In this broad research area, the key question usually being asked is: how dense an object could be such that it \textit{avoids} a \textit{forbidden} or \textit{excluded} object. In this paper, the object of interest is \textit{multidimensional 0-1 matrix}. A matrix is called a 0-1 matrix if all its entries are either $0$ or $1$.
We say that a 0-1 matrix $A$ \textit{contains} another 0-1 matrix $P$ if $A$ has a submatrix that can be transformed to $P$ by changing any number of $1$-entries to $0$-entries. Otherwise, $A$ \textit{avoids} $P$. With this context, the density of concern is the number of $1$-entries of a 0-1 matrix, which is sometimes called its \textit{weight}. Following these, the key question can be formulated as seeking the asymptotic behavior of function $ex(n_1,n_2,\ldots,n_d;P)$, defined as the maximum weight of a $d$-dimensional $n_1\times n_2\times\ldots\times n_d$ 0-1 matrix that avoids another $d$-dimensional 0-1 matrix $P$. This problem can be seen as finding the maximum possible weight  of a $d$-dimensional $n_1\times n_2\times\ldots\times n_d$  0-1 matrix $A$ that is \textit{saturating} for 0-1 matrix $P$ i.e. $A$ does not contain $P$ and turning any $0$-entry of $A$ to $1$-entry introduces a copy of $P$ in $A$. It is then natural to also ask for the \textit{minimum} possible weight of a $d$-dimensional $n_1\times n_2\times\ldots n_d$ 0-1 matrix saturating for $P$, denoted $sat(n_1,n_2,\ldots,n_d;P)$. A variation of saturation is \textit{semisaturation}: $A$ is semisaturating for $P$ if flipping any $0$-entry of $A$ to $1$-entry creates a \textit{new} copy of $P$ in $A$. The corresponding extremal function is denoted $ssat(n_1,n_2,\ldots,n_d;P)$: the minimum possible weight of a $d$-dimensional $n_1\times n_2\times\ldots\times n_d$ 0-1 matrix semisaturating for $P$. When $n_1=n_2=\ldots=n_d$, we use the simplified notations $ex(n;P,d)$, $sat(n;P,d)$, and $ssat(n;P,d)$. By definition
$$
ssat(n_1,n_2,\ldots,n_d;P)\leq sat(n_1,n_2,\ldots,n_d;P)\leq ex(n_1,n_2,\ldots,n_d;P).
$$
In this paper we present two major results. First, for $P$ as a $d$-dimensional identity matrix we give the exact value of functions $ex(n_1,n_2,\ldots,n_d;P)$ and  $sat(n_1,n_2,\ldots,n_d;P)$, which are shown to be identical. This together with the implied structure of $d$-dimensional 0-1 matrices saturating for $P$ generalize the two-dimensional result by Brualdi and Cao \cite{BC2021} and Tsai's discovery that every maximal antichain in a strict chain product poset is also maximum \cite{Tsai2020}. Second, as a partial extension to Fulek and Keszegh's constant versus linear dichotomy for semisaturation function of two-dimensional 0-1 matrices  \cite{FK2021}, we give the necessary and sufficient condition for a $d$-dimensional 0-1 matrix to have bounded semisaturation function.\\

In Section~\ref{section:related-works} we review previous works relevant to our study. Section~\ref{section:notations} contains terminologies used throughout the paper. Our results and proofs are in Section~\ref{section:our-results}.

\section{Related works}\label{section:related-works}
The extremal theory of 0-1 matrix started around 1990 in studies of computational and discrete geometry problems. Mitchell presented an algorithm for finding the shortest $L_1$ path between two
points in a rectilinear grid with obstacles \cite{mitchell1992}. Its time complexity was bounded above via certain matrices' extremal function $ex(n;P)$ given by Bienstock and Gy\H{o}ri \cite{BG1991}. In 1959, Erd\H{o}s and Moser asked for the maximum number of unit distances among the vertices of a convex $n$-gon \cite{EM1959}. In 1990, F\"uredi gave the first upper bound $O(n\log_2 n)$ that is tighter than $n^{1+\epsilon}$ via the extremal theory \cite{furedi1990maximum}. Pach and Sharir used extremal functions $ex(n;P)$ to bound the number of pairs of non-intersecting and vertically visible line segments \cite{pach1991vertical}. One of the latest applications is the resolution of Stanley-Wilf Conjecture in enumerative combinatorics \cite{klazar2000,MT2003} in 2004 as Marcus and Tardos showed that every two-dimensional permutation matrix's extremal function $ex(n;P)$ is linear \cite{MT2003}.
After applying to geometry-related problems mentioned previously, F\"uredi and Hajnal \cite{FH1992} and Tardos \cite{Tardos2005} asymptotically decide the extremal functions $ex(n;P)$ for every 0-1 matrix with no more than four $1$-entries.\\

The extremal theory also extends from two-dimensional to multidimensional 0-1 matrices. Extending Marcus and Tardos' result above on two-dimensional permutation matrix \cite{MT2003}, Klazar and Marcus proved that the extremal function $ex(n;P)$ of a $d$-dimensional $k\times\ldots\times k$ permutation matrix is $O(n^{d-1})$ \cite{KM2007}. Geneson and Tian gave nontrivial bounds on the extremal function of \textit{block permutation matrix} i.e. Kronecker product of permutation matrix and \textit{block matrix} with no $0$-entry \cite{GT2017}, extending Geneson's result on two-dimensional tuple permutation matrix \cite{Geneson2009}. In another direction, they substantially improved the limit inferior and limit superior of the sequence $\frac{ex(n;P)}{n^{d-1}}$ for tuple permutation matrices and permutation matrices.\\

Recently, Brualdi and Cao initiated the study of saturation problem for two-dimensional 0-1 matrices \cite{BC2021}. They proved that every maximal matrix avoiding the identity matrix $I_k$ have the same weight. Fulek and Keszegh found a general upper bound on the saturation function in terms of the dimensions of $P$, and showed that the saturation function is either bounded or linear \cite{FK2021}. Geneson found that almost all permutation matrices have bounded saturation function \cite{Geneson2021}, followed by Berendsohn's full characterization of permutation matrices with bounded saturation function \cite{Berendsohn2021}.

\section{Notations}\label{section:notations}
For positive integer $d$, let $[d]$ denote $\{1,2,\ldots,d\}$. We denote a $d$-dimensional $n_1\times n_2\times\ldots\times n_d$ matrix by $A=(a_{x_1,\ldots,x_d})$, where $x_i\in[n_i]$ for each $i\in [d]$. 
A \textit{k-dimensional cross section} $L$ of a $d$-dimensional $n_1\times n_2\times\ldots\times n_d$ matrix $A$ is the set of all entries of $A$ whose coordinates of a fixed set $C_L$ of $d-k$ dimensions are fixed. A cross section $L$ of matrix $A$ is a \textit{face} if for every $i\in C_L$, the value of the $i^{\text{th}}$ coordinate is fixed to some $p_i\in\{1,n_i\}$.
An $i$-\textit{row} of matrix $A$ is a cross section $L$ with $C_L=[d]-\{i\}$.\\

Let $z$ and $o$ be $0$-entry and $1$-entry of 0-1 matrices $M$ and $P$, respectively. If flipping $z$ to $1$ introduces a new copy of $P$ in which the new $1$-entry matches $o$, then we say $z$ \textit{potentially matches} $o$.\\

Given a $d$-dimensional matrix $A$, entries $a_{x_1,\ldots,x_d}$ and $a_{y_1,\ldots,y_d}$ belong to the same \textit{diagonal} if $x_1-y_1=x_2-y_2=\ldots=x_d-y_d$. Diagonal is same as \textit{shape} in \cite{Tsai2020}. An $n_1\times n_2\times\ldots n_d$ matrix has $\prod_{i\in[d]}n_i-\prod_{i\in[d]}(n_i-1)$ diagonals. An entry $a_{x_1,\ldots,x_d}$ is \textit{below} or \textit{above} another entry $a_{y_1,\ldots,y_d}$ if $x_i>y_i$ for every $i\in[d]$ or $x_i<y_i$ for every $i\in[d]$, respectively. An entry $a_{x_1,\ldots,x_d}$ is \textit{semibelow} or \textit{semiabove} another entry $a_{y_1,\ldots,y_d}$ if $x_i\geq y_i$ for every $i\in[d]$ or $x_i\leq y_i$ for every $i\in[d]$, respectively. Two entries are \textit{comparable} if one of them is above the other. A $d$-dimensional \textit{staircase} generalized from \cite{FK2021} and \textit{zigzag path} in \cite{BC2021} is a set of pairwise incomparable entries in $A$. When the set is maximal, it is denoted \textit{complete} staircase. Note that staircase in \cite{FK2021} means complete staircase here.  From \cite{Tsai2020}, every complete staircase has size $\prod_{i\in[d]}n_i-\prod_{i\in[d]}(n_i-1)$ and intersects with every diagonal at exactly one entry. An entry $e_1$ is \textit{above} or \textit{below} a complete staircase $S$ if $S$ has an entry $e_2$ such that $e_1$ is above or below $e_2$, respectively. Every entry in $A$ is either above, below, or part of any complete staircase $S$ of $A$. A \textit{shell} is the unique complete staircase containing the corner entry $a_{n_1,\ldots,n_d}$. Similar to 0-1 matrices, the weight of a staircase is the number of $1$-entries it has.\\

\section{Our results}\label{section:our-results}
The following extension from Lemma 3.3 of \cite{FK2021} is handy in dealing with $d$-dimensional identity matrix.
\begin{lemma}\label{lemma:diagonal}
Let $P$ be a $d$-dimensional $l_1\times\ldots\times l_d$ 0-1 matrix with $l_i>1$ for each $i\in[d]$. The only $1$-entry in $P$'s shell is the corner $p_{l_1,\ldots,l_d}$. Then in any $d$-dimensional 0-1 matrix $M$ saturating for $P$ there is a $d$-dimensional complete staircase $S$ that is all $1$, and all entries below $S$ are $0$-entries that potentially match $p_{l_1,\ldots,l_d}$.
\end{lemma}
\begin{proof}
We claim that the set of bottommost $1$-entries from all diagonals of $M$ forms a $d$-dimensional complete staircase.
Observe that if $z_1,z_2,\ldots,z_k$ are consecutive $0$-entries of a diagonal $D$ of $M$ where $z_1$ is the bottom entry of $D$, then they all potentially match $p_{l_1,\ldots,l_d}$. This implies every diagonal of $M$ contains at least an $1$-entry, because the top entry of any diagonal does not potentially match $p_{l_1,\ldots,l_d}$. We are left to prove that no bottommost $1$-entry $o_1$ of a diagonal $D_1$ is above the bottommost $1$-entry $o_2$ of any other diagonal $D_2$. If $o_1$ is above $o_2$, then let $z_1$ be the $0$-entry in $D_1$ immediately below $o_1$. Entry $z_1$ potentially matches $p_{l_1,\ldots,l_d}$ and is semiabove $o_2$, creating a copy of $P$ in $M$ where $o_2$ matches $p_{l_1,\ldots,l_d}$.
\end{proof}

Let $A$ and $B$ be $d$-dimensional $l_1\times l_2\times\ldots\times l_d$ and $k_1\times k_2\times\ldots\times k_d$ 0-1 matrices, respectively. The \textit{diagonal concatenation} $M$ of $A$ and $B$ is a $d$-dimensional $(l_1+k_1)\times (l_2+k_2)\times\ldots\times(l_d+k_d)$ 0-1 matrix, such that $m_{x_1,x_2,\ldots,x_d}=a_{x_1,x_2,\ldots,x_d}$ if $x_i\leq l_i$ for every $i\in[d]$, $m_{x_1,x_2,\ldots,x_d}=b_{x_1-l_1,x_2-l_2,\ldots,x_d-l_d}$ if $l_i+1\leq x_i$ for every $i\in[d]$, and $m_{x_1,\ldots,x_d}=0$ otherwise.\\

    The Lemma below generalizes Theorem 1.9 of \cite{FK2021}.
\begin{lemma}\label{lemma:add}
Let $A$ be a $d$-dimension $l_1\times l_2\times\ldots\times l_d$ 0-1 matrix, $I$ be a $d$-dimensional $1\times\ldots\times 1$ identity 0-1 matrix, $P'$ be the diagonal concatenation of $A$ and $I$, and $P$ be the diagonal concatenation of $P'$ and $I$. Then
\begin{align*}
&sat(n_1,n_2,\ldots,n_d;P)=sat(n_1-1,n_2-1,\ldots,n_d-1;P')+\prod_{i\in[d]}n_i-\prod_{i\in[d]}(n_i-1)\\
&ex(n_1,n_2,\ldots,n_d;P)=ex(n_1-1,n_2-1,\ldots,n_d-1;P')+\prod_{i\in[d]}n_i-\prod_{i\in[d]}(n_i-1)
\end{align*}
where $l_i+2\leq n_i$ for each $i\in[d]$.
\end{lemma}
\begin{proof}
Given a $d$-dimensional $n_1\times n_2\times\ldots n_d$ 0-1 matrix $M$, let $S$ denote the shell of $M$ and let $M'$ denote the $(n_1-1)\times (n_2-1)\times\ldots (n_d-1)$ submatrix of $M$ that does not contain any entry of $S$. We show that if $S$ is all $1$ and $M'$ is saturating for $P'$, then $M$ is saturating for $P$. This implies $sat(n_1,n_2,\ldots,n_d;P)\leq sat(n_1-1,n_2-1,\ldots,n_d-1;P')+\prod_{i\in[d]}n_i-\prod_{i\in[d]}(n_i-1)$ and $ex(n_1,n_2,\ldots,n_d;P)\geq ex(n_1-1,n_2-1,\ldots,n_d-1;P')+\prod_{i\in[d]}n_i-\prod_{i\in[d]}(n_i-1)$. Indeed, $M$ clearly avoids $P$. All $0$-entries of $M$ lie in $M'$, so flipping any of them creates a copy $P''$ of $P'$ in $M'$. The copy $P''$ and $S$ filled with $1$-entries form a copy of $P$, so $M$ is saturating for $P$.\\

We are left to prove that $sat(n_1,n_2,\ldots,n_d;P)\geq sat(n_1-1,n_2-1,\ldots,n_d-1;P')+\prod_{i\in[d]}n_i-\prod_{i\in[d]}(n_i-1)$ and $ex(n_1,n_2,\ldots,n_d;P)\leq ex(n_1-1,n_2-1,\ldots,n_d-1;P')+\prod_{i\in[d]}n_i-\prod_{i\in[d]}(n_i-1)$. Given a matrix $N$ saturating for $P$, apply Lemma~\ref{lemma:diagonal} to get a $d$-dimensional complete staircase $T$ of $N$ that is all $1$ and all entries below $T$ are $0$-entries that potentially match $p_{l_1,\ldots,l_d}$. Obtain a $d$-dimensional matrix $N'$ by changing every $1$-entry of $T$ to $0$-entry followed by removing the shell of $N$. It suffices to show that $N'$ is saturating for $P'$.  Matrix $N'$ clearly avoids $P'$. For clarity denote by $T'$ the entries of $N'$ with the same coordinates as $T$, as $T$ intersects with the removed shell of $N$. Entries of $T'$ also form a $d$-dimensional complete staircase of $N'$. It could be seen that every $0$-entry of $N'$ that is below or part of $T'$ potentially matches the corner $p'_{l_1-1,\ldots,l_d-1}$ of $P'$. Moreover every $0$-entry of $N'$ above $T'$ potentially matches some $1$-entry of $P'$, because in $N$ together with $T$  turning any of them to $1$-entry introduces a copy of $P$.
\end{proof}

By successive application of Lemma~\ref{lemma:diagonal} and  Lemma~\ref{lemma:add} to $d$-dimensional $(k+1)\times(k+1)\times\ldots\times(k+1)$ identity matrix, we obtain the following generalization of Brualdi and Cao's result on 0-1 matrices saturating for identity matrices \cite{BC2021}. It also extends Tsai's discovery that every maximal antichain in a strict chain product poset is also maximum \cite{Tsai2020}, and implies that the greedy algorithm to generate 0-1 matrices saturating for identity matrix \cite{BC2021} also works for multidimensional matrices.
\begin{theorem}
Let $P$ be a $d$-dimensional $(k+1)\times(k+1)\times\ldots\times(k+1)$ identity matrix. Suppose $k\leq n_i$ for each $i\in[d]$ and $k\leq n$. Then
\begin{align*}
&sat(n_1,n_2,\ldots,n_d;P)=ex(n_1,n_2,\ldots,n_d;P)=\prod_{i\in[d]}n_i-\prod_{i\in[d]}(n_i-k)\\
&sat(n;P,d)=ex(n;P,d)=n^d-(n-k)^d=\Theta(n^{d-1}).
\end{align*}
Moreover, the $1$-entries of every $d$-dimensional $n_1\times n_2\times\ldots\times n_d$ 0-1 matrix saturating for $P$ can be decomposed into $k$ $d$-dimensional staircases with weights $\prod_{i\in[d]}n_i-\prod_{i\in[d]}(n_i-1)$, $\prod_{i\in[d]}(n_i-1)-\prod_{i\in[d]}(n_i-2)$, $\ldots$, $\prod_{i\in[d]}(n_i-k+1)-\prod_{i\in[d]}(n_i-k)$, respectively. 
\end{theorem}

All remaining results are extended from Fulek and Keszegh's work on two-dimensional 0-1 matrices \cite{FK2021}. Note that the modified conditions in Lemma~\ref{lemma:only} and Theorem~\ref{theorem:bounded-ssat} are not trivial.

\begin{lemma}
For any $d$-dimensional 0-1 matrix $P$, $sat(n;P,d)=O(n^{d-1})$.
\end{lemma}
\begin{proof}
If $P$ does not contain an $1$-entry, then $sat(n;P,d)=0$. Otherwise let the dimensions of $P$ be $l_1\times l_2\times\ldots\times l_d$, and let $a_{x_1,x_2,\ldots,x_d}$ be an $1$-entry.
Construct a $d$-dimensional $n\times n\times\ldots\times n$ 0-1 matrix $M$ saturating for $P$ as follows. For any entry $m_{x_1',x_2',\ldots,x_d'}$, set it to $0$ if and only if $x_i\leq x'_i\leq n-(l_i-x_i)$ for each $i\in[d]$. Clearly $M$ avoids $P$ and turning any $0$-entry of $M$ into $1$ makes it contain $P$. The weight of $M$ is $n^d-(n-l_1+1)(n-l_2+1)\ldots(n-l_d+1)=O(n^{d-1})$.
\end{proof}

\begin{lemma}
Let $P$ be the diagonal concatenation of non-zero $d$-dimensional 0-1 matrices $A$ and $B$. Then $sat(n;P,d)=\Theta(n^{d-1})$.
\end{lemma}
\begin{proof}
We show by contradiction that every $i$-row of a $d$-dimensional 0-1 matrix $M$ saturating for $P$ contains at least an $1$-entry where $i\in[d]$. Suppose $M$ has an $i$-row $r$ which is all $0$. Flipping the first or last entry of $r$ to $1$ creates a copy of $P$ where the new $1$-entry matches an $1$-entry of $A$ or $B$, respectively. Thus there exist two adjacent entries $z_1$ and $z_2$ of $r$ such that flipping $z_1$ to $1$ creates a copy $P'$ of $P$ where $z_1$  matches an $1$-entry of $A$, and flipping $z_2$ to $1$ creates a copy $P''$ of $P$ where $z_2$ matches an $1$-entry of $B$. The copy of $B$ in $P'$ and the copy of $A$ in $P''$ form a copy of $P$, i.e. $M$ contains $P$.
\end{proof}

\begin{lemma}\label{lemma:only}
Let $d'<d$. 
If a non-zero $d$-dimensional 0-1 matrix $P$ does not have an $1$-entry which is the only $1$-entry in every $d'$-dimensional cross section of $P$ it belongs to, then $ssat(n;P,d)=\Omega(n^{d-d'})$.
\end{lemma}
\begin{proof}
Suppose $M$ is semisaturating for $P$. Say a $0$-entry and an $1$-entry of $M$ are \textit{connected} if they are in the same $d'$-dimensional cross section of $M$. Each $0$-entry is connected with at least an $1$-entry, and each $1$-entry is connected with at most $\binom{d}{d'}(n^{d'}-1)$ $0$-entries. So the weight of $M$ is at least 
$$
\frac{n^d}{1+\binom{d}{d'}(n^{d'}-1)}=\Theta(n^{d-d'}).
$$
\end{proof}

\begin{theorem}\label{theorem:bounded-ssat}
Given a $d$-dimensional pattern $P$, $ssat(n;P,d)=O(1)$ if and only all the following properties hold for $P$:\\
(i) For any $d'\in[d-1]$, every $d'$-dimensional face $f$ of $P$ contains an $1$-entry $o$ that is the only $1$-entry in every $(d-1)$-dimensional cross section that is orthogonal to $f$ and intersects $f$ at $o$.\\
(ii) $P$ contains an $1$-entry that is the only $1$-entry in every $(d-1)$-dimensional cross section it belongs to.
\end{theorem}
\begin{proof}
Let $M$ be a $d$-dimensional 0-1 matrix saturating for $P$. Suppose $P$ does not have property (i) i.e. for some $d'\in[d-1]$ $P$ has a $d'$-dimensional face $f$ that does not contain an $1$-entry $o$ that is the only $1$-entry in every $(d-1)$-dimensional cross section that is orthogonal to $f$ and intersects $f$ at $o$. Let the counterpart of $f$ in $M$ be $f'$. If $f$ is all $0$, then $f'$ must be all $1$. Otherwise, we say a $0$-entry in $f'$ is connected to an $1$-entry in $M-f'$ if they are in the same $(d-1)$-dimensional cross section. Each $0$-entry of $f'$ is connected to at least an $1$-entry. Each $1$-entry of $M-f'$ is in $d'$ $(d-1)$-dimensional cross sections that is orthogonal to $f'$, and each of them contains at most $n^{d'-1}$ $0$-entries of $f'$. Thus each $1$-entry of $M-f'$ is connected to at most $d'(n^{d'-1})$ $0$-entries in $f'$. Suppose $f'$ has $\alpha$ $0$-entries, then $M$ has at least $(n^{d'}-\alpha)+\frac{\alpha}{d'(n^{d'-1})}\geq \frac{n^{d'}}{d'(n^{d'-1})}=\Theta(n)$ $1$-entries. Otherwise suppose $P$ does not have property (ii) i.e. it does not contain an $1$-entry that is the only $1$-entry in every $(d-1)$-dimensional cross section it belongs to. By Lemma~\ref{lemma:only} $ssat(n;P,d)=\Omega(n)$.\\

We are left to prove that $ssat(n;P,d)=O(1)$ if $P$ has all the properties. Suppose $P$'s dimensions are $l_1\times l_2\times\ldots\times l_d$. We construct an $n\times n\times\ldots\times n$ 0-1 matrix $M'$ semisaturating for $P$ with $O(1)$ $1$-entries as follows. Entry $m'_{l_1',l_2',\ldots,l_d'}$ is $1$ if and only if for each $i\in[d]$, $l_i'<l_i$ or $n+1-l_i<l_i'$.
Let an $1$-entry of $P$ with property (ii) be $p_{o_1,o_2,\ldots,o_d}$.
For a given $0$-entry $m'_{z_1,z_2,\ldots,z_d}$, suppose turning it to $1$ produces another matrix $M''$. If $z_i\in[l_i, n+1-l_i]$ for each $i\in[d]$ then for each dimension $i$ we slice $M''$ by indices  $\{1,2,\ldots,o_i-1,z_i,n-l_i+o_i+1,n-l_i+o_i+2,\ldots,n\}$, and the resulting $l_1\times\ldots\times\l_d$ submatrix contains $P$.\\

If $z_i\notin [l_i,n+1-l_i]$ for some $i\in[d]$, there exists a partitioning $[d]=X\cup Y\cup Z$ such that $X\cup Y\ne \emptyset$ and $\forall i\in[d]$
$$
\begin{cases}
z_i<l_i, & \text{if }i\in X\\
n+1-l_i<z_i, & \text{if }i\in Y\\
z_i\in[l_i, n+1-l_i], &\text{otherwise}
\end{cases}.
$$
Let $f$ be the face of $P$ whose $i^{\text{th}}$ coordinate is fixed to $1$ or $l_i$ if $i\in X$ or $i\in Y$, respectively, and for each $i\in Z$ let the $i^{\text{th}}$ coordinate of the $1$-entry contained in $f$ with property (i) be $l_i''$. We then slice $M''$ by indices $I_i$ for each $i\in [d]$ to obtain a $l_1\times\ldots\l_d$ submatrix containing $P$: for each $i\in X$, $I_i=\{z_i,n-l_i+2,n-l_i+1,\ldots,n\}$; for each $i\in Y$, $I_i=\{1,2,\ldots,l_i-1,z_i\}$; otherwise $I_i=\{1,2,\ldots,l_{i}''-1, z_i, n-l_i+l_{i}''+1,\ldots,n\}$.
\end{proof}
\bibliographystyle{plain}
\bibliography{pattern-avoidance}

\end{document}